\documentclass[12pt,reqno]{amsart}

\usepackage{color}
\usepackage{amsmath, amsthm, amssymb}
\usepackage{amsfonts}
\usepackage[ansinew]{inputenc}
\usepackage[dvips]{epsfig}
\usepackage{graphicx}
\usepackage[english]{babel}
\usepackage{hyperref}

\usepackage{cite}
\usepackage{graphicx}
\usepackage{amscd}
\usepackage{color}
\usepackage{bm}
\usepackage{enumerate}

\usepackage{verbatim}
\usepackage{hyperref}
\usepackage{amstext}
\usepackage{latexsym}
\theoremstyle{plain}
\newtheorem{theorem}{Theorem}[section]

\newtheorem{corollary}[theorem]{Corollary}

\newtheorem{lemma}[theorem]{Lemma}
\newtheorem{remark}[theorem]{Remark}

\numberwithin{theorem}{section}
\numberwithin{equation}{section}

\newcommand{\average}{{\mathchoice {\kern1ex\vcenter{\hrule height.4pt
width 6pt depth0pt} \kern-9.7pt} {\kern1ex\vcenter{\hrule
height.4pt width 4.3pt depth0pt} \kern-7pt} {} {} }}

\def\R{\mathbb{R}}

\renewcommand{\a }{\alpha }
\renewcommand{\b }{\beta }
\renewcommand{\d}{\delta }

\newcommand{\D }{\Delta }

\newcommand{\e }{\varepsilon }
\newcommand{\g }{\gamma}

\newcommand{\G }{\Gamma}

\newcommand{\vp }{\varphi }

\renewcommand{\O }{\Omega }

\newcommand{\be}{\begin{equation}}
\newcommand{\ee}{\end{equation}}

\newcommand{\de}{\partial}

\newcommand{\ti}{\widetilde}

\newcommand{\ra}{{\rangle}}
\newcommand{\la}{{\langle}}

\newcommand{\calL }{\mathcal{L}}

\newcommand{\calD }{\mathcal{D}}

\newcommand{\N}{\mathbb{N}}

\renewcommand{\epsilon}{\varepsilon}

\newcommand{\Ds}{ (-\D)^s}

\begin{document}

\title 
{Entire $s$-harmonic functions are affine}

\author[Mouhamed M. Fall]
{Mouhamed Moustapha Fall}
\address{M.M.F.: African Institute for Mathematical Sciences in Senegal, 
KM 2, Route de Joal, B.P. 14 18. Mbour, S\'en\'egal}
\email{mouhamed.m.fall@aims-senegal.org, mouhamed.m.fall@gmail.com}

\thanks{This work  is supported by the Alexander von Humboldt foundation and the author would
like to thank Tobias Weth and Krzysztof Bogdan for useful discussions. This work was completed while the author was visiting the Goethe-Universit\"at Frankfurt am Main and the Technische Universit\"at Chemnitz.
He is also very grateful to the referee for his/her detailed comments. 
The variety of
his/her substantial suggestions  helped him to improve  the first versions of this manuscript.
}


 \begin{abstract}
   \noindent
   In this paper, we prove that solutions to the equation $(-\D)^s u=0$ in $\mathbb{R}^N$, for $s\in (0,1)$, are affine. This  allows us 
   to  prove the uniqueness of the   Riesz potential $|x|^{2s-N}$ in Lebesgue spaces.
 \end{abstract}

\maketitle



%
 \section{Introduction}
The classical Liouville theorem for harmonic functions states that \textit{ a bounded harmonic function in $\R^N$ is constant}, see for instance the particularly short proof by E. Nelson in \cite{EN}. The stronger version of it states that a nonnegative  harmonic function on $\R^N$ is constant.   In the case of the fractional Laplacian $-\Ds$, for $s\in(0,1)$, (see Section \ref{s:prem}), the strong form of the Liouville theorem holds as well and was proved by K. Bogdan et al. in \cite{BKN}. Applications   of   Liouville theorems for nonlocal operators in the study of 
nonlocal elliptic systems of equations can be found in \cite{RWXZ,R-O-S, FV, Abat}. \\
 The aim of this paper is to classify  all $s$-harmonic functions in $\R^N$,   thereby obtaining the  Liouville theorem for the fractional Laplacian as a particular case. 
%
%
\begin{theorem}\label{th:ClassHArm}
Every $s$-harmonic function  in $\R^N$ is affine, and constant if $s\in(0,1/2]$.
\end{theorem}
The proof of this theorem is mainly based on  a Cauchy-type estimate for the derivatives of an  $s$-harmonic function. More precisely, given $s\in(0,1)$, $\g\in\N^N$  and a function $u$ which is $s$-harmonic  in the ball $B(0,R)$, we have the estimate
\be\label{eq:Cauchy-est} 
|D^\g u(0)|\leq C  R^{2s-|\g|}\int_{|y|\geq R/4}|u(y)||y|^{-N-2s}\, dy,
\ee
for some positive constant $C $ depending only on $N,\g$ and $s$, see Section \ref{s:Proofs}. This estimate is obtained from the Poisson kernel  representation formula for $s$-harmonic functions. We refer to  Section \ref{s:prem} for more  details.\\
  In the following, we denote by  $\calD'(\R^N)$  the dual of $C^\infty_c(\R^N)$ endowed with the usual topology.
An iteration argument based on  {Theorem} \ref{th:ClassHArm} allows to state the following result. 
\begin{theorem}\label{th:Liouvilles012Poly}
Assume that $s\in (0,1)$ and let $ u $  be a solution to the equation
$$
\Ds u= P \quad \textrm{ in } \calD'(\R^N),
$$
where $P$ is a polynomial.
 Then $u$ is affine and  $P=0$.
\end{theorem}
 Another consequence of the main theorem which is of independent
interest is the following result.
\begin{corollary}  \label{cor:cor0}
Let $p\in[1,\infty)$ and $u\in L^p(\R^N)$ be  such that
$$
\Ds u=0 \quad \textrm{ in } \calD'(\R^N).
$$
Then  $u\equiv 0$.
\end{corollary}
Combining  {Corollary}  \ref{cor:cor0} and  the Hardy-Littlewood-Sobolev inequality, we have
a uniqueness result.
\begin{corollary}[Uniqueness of Riesz potential] \label{th:uniRiesz}
Let $s\in(0,1)$, $ 1<p<\frac{N}{2s}$ and  $f\in L^p(\R^N)$.
Then there exists a unique $u\in L^{\frac{Np}{N-2sp}}(\R^N)$ such
that
$$
\Ds u=f \quad \textrm{ in } \calD'(\R^N),
$$
and $u$ is given by
$$
u(x)= \a_{N,s}\int_{\R^N}\frac{f(y)}{|x-y|^{N-2s}}\, dy,
$$
where
$$
\a_{N,s}=\pi^{N/2} 2^{2s}\frac{\G(s)}{\G( (N-2s)/2)}.
$$
\end{corollary}
The paper is organized as follows. In  Section \ref{s:prem} we   collect some basic facts concerning the fractional Laplacian $-\Ds$  and $s$-harmonic functions. Finally, in Section \ref{s:Proofs}    we prove the Cauchy-type estimate \eqref{eq:Cauchy-est}, the main result and its corollaries.
\bigskip

\noindent
\textbf{Note added in proof}:  We  mention that after this paper was submitted, Liouville-type  results for a class of  nonlocal operators   were proved in \cite{CDL} and \cite{FW-Liouville} using  Fourier transform.

\section{Preliminaries}\label{s:prem}
This section is devoted to recall some basic notions about $s$-harmonic functions. We refer the reader to \cite[Section 3]{BB-Schr}.   Let $\calL^1_s$ denote the
space of all measurable functions $u:\R^N\to
\R$ such that $$\int_{\R^N}\frac{|u(x)|}{1+|x|^{N+2s}}dx<\infty.$$
For functions
$\vp \in C^2(\R^N)\cap\calL^1_s$, the fractional Laplacian $-(-\Delta)^s$ is defined by
\begin{equation}
  \label{eq:3}
-\Ds \vp(x)=C_{N,s}\lim_{\e\to0}\int_{|x-y|>\e}\frac{\vp(y)-\vp(x)}{|y-x|^{N+2s}}\,dy \qquad\textrm{ for all } x\in\R^N,
\end{equation}
where $C_{N,s}=s(1-s)\pi^{-N/2}4^s\frac{\Gamma(\frac{N}{2}+s)}{\Gamma(2-s)}$.\\
For $u\in \calL^1_s$, the expression $\Ds u$ defines a distribution on every open set $\O\subset\R^N$ by 
$$
\la \Ds u, \vp \ra =\int_{\R^N}u(x)\Ds\vp(x)\, dx  \qquad \textrm{ for every $\vp\in C^\infty_c(\O)$.} 
$$
In the case where $\Ds u= 0$ in $\calD'(\O)$,   we will say that    \textit{$u$ is $s$-harmonic} in $\O$.\\
%
%
%
We note that affine functions $u$  belong to $\calL^1_s$ if $s>1/2 $ and constant functions $u$ belong to $\calL^1_s$ if  $s\in(0,1/2]$. Moreover, by using   \eqref{eq:3}, in both cases, we can see that $\Ds u (x)=0$ for every $x\in\R^N$.  Furthermore, thanks to  \cite[Lemma 3.3]{BB-Schr}, we have $\Ds u=0$ in $\calD'(\R^N)$. \\
%
The fractional Laplacian has an explicit Poisson kernel with respect to the ball $B(x,r)$ (see \cite{BGR}). It  is given by
\be\label{eq:PK}
P_r(x,y)= \begin{cases} \displaystyle \b_{N,s}
 \frac{(r^2-|x|^2)^s}{(|y|^2-r^2)^s}|y-x|^{-N}\quad \textrm{ for }
 |x|< r,\,|y|>r, \vspace{3mm}\\
\displaystyle  0  \qquad \qquad \textrm{otherwise},
 \end{cases}
 \ee
 where $\b_{N,s}=\G(N/2)\pi^{-N/2-1}\sin(s\pi)$. Therefore (see also \cite{BKN}), if $u$ is  $s$-harmonic  in $\O$ then for every  ball $B(a,r)\subset\subset\O$ we have 
 $$
 u(x)=\int_{\R^N}P_r(x-a,y-a) u(y)\, dy \qquad \textrm{ for all } x\in B(a,r).
 $$
We now consider the regularization
of $P_r$ as   in \cite{BB-Schr}. To this end, we pick a   function $\phi\in C^\infty_c(1,4)$  such that  $\int_{\R}\phi(r)\,dr =1$ and define $\Psi: \R^N\to \R$ by  
$$
\Psi(y)=
\int_{1}^{4}{P_r(0,y)} \phi(r)\,dr=\b_{N,s}
|y|^{-N}\int_{\min(1,|y|)}^{\min(4,|y|)}
r^{2s}(|y|^2-r^2)^{-s}\phi(r)\,dr.
$$
  Observe that, if $|y|\leq 1$ then $(0,|y|)\cap (1,4)=\emptyset$ and thus  
\be \label{eq:Psieq0inBr0}
\textrm{$\Psi(y)=0 \qquad $ for every $y\in B(0,1) $.}
\ee
Furthermore, as shown e.g. in  \cite[Lemma 3.11]{BB-Schr}, we have $\Psi\in C^\infty(\R^N)$. Moreover, for every      $\g\in\N^N$ there holds   
\be 
\label{eq:estPisr0}
|D^{\g}\Psi (y)|\leq{C}\,{|y|^{-N-2s-|\g|}}\qquad \textrm{ for every } y\in \R^N\setminus\{0\},
\ee
where $C=C(N,\g,s)$  denotes, here and in the following, a positive constant depending only on $N$, $\g$ and $s$.\\
 We define $\Psi_{r_0}(y)=r_0^{-N}\Psi(y/r_0)$, for $y\in\R^N$ and $r_0>0$. Then, for any  $s$-harmonic function  $u$  in an open set $\O\subset\R^N$, we have   
\be \label{eq:uequstarPsi}
u(x)=u\star \Psi_{r_0}(x) \qquad \textrm{ for all almost every } x\in \O_{4r_0},
\ee
where  $\O_{4r_0}=\left\{x\in\O\,:\,\textrm{dist}(x,\R^N\setminus\O)>4r_0\right\}$,
 see    \cite[Lemma 2.6]{FW-half} or \cite[Page 65]{BB-Schr}.  We will therefore assume, in the sequel, that  $s$-harmonic functions in some open set are  smooth in that set.
%
\section{Proof of the main result and its consequences}\label{s:Proofs}
The following result (from which we will derive our main result)  can be seen as a nonlocal version of the Cauchy estimate   for bounded harmonic functions, see e.g. \cite[Chapter 2]{ABR}. 
\begin{lemma}\label{lem:Est-harm}
 For every $\g\in\N^N$, there exists a constant $C>0$ only depending on  $N$, $\g$ and $s$ such that  for every function  $u$ which is  $s$-harmonic in $B(0,R)$,
$$
|D^{\g} u(0)|\leq C \, R^{2s-|\g|}\,\int_{|y|\geq  R/4}{|u(y)|}{|y|^{-N-2s}}\, dy.
$$
\end{lemma}
\begin{proof}
 Let $u$ be  an $s$-harmonic function in $B(0,R)$ and $r_0\in (0,R/4)$. Then by \eqref{eq:uequstarPsi} we have
$$ 
u(x)=u\star \Psi_{r_0}(x) \qquad \textrm{ for all } x\in B(0, R-4r_0).
$$
 By \eqref{eq:Psieq0inBr0}, \eqref{eq:estPisr0} and the   dominated convergence theorem, we deduce that  
$$
D^\g u(0)=u\star D^\g  \Psi_{r_0}(0) \qquad \textrm{ for all }   \g\in \N^N.
$$
Using once more \eqref{eq:Psieq0inBr0} and  \eqref{eq:estPisr0}, we get 
\begin{align*}
|D^\g  u(0)|= \left| \int_{|y|\geq r_0 }u(y)D^\g\Psi_{r_0}(-y)\, dy\right|\leq C \,r_0^{2s}\int_{|y|\geq r_0 }
|u(y)|\, |y|^{-N-2s-|\g|}\,dy.
\end{align*}
It follows that 
\begin{align*}
|D^\g  u(0)| \leq C \,r_0^{2s-|\g|}\int_{|y|\geq r_0 }
|u(y)|\, |y|^{-N-2s}\,dy.
\end{align*}
Letting $r_0\to  R/4$, we get  the desired estimate.
\end{proof}
%
%
As a consequence of  Lemma \ref{lem:Est-harm}, we have the following result.
\begin{corollary}\label{Cor:est_u}
Let $\O$ be a nonempty open set of $\R^N$ such that $\O\neq\R^N$. Then   for every $\g\in\N^N$, there exists a constant $C>0$ only depending on  $N$, $\g$ and $s$ such that  for every  function $u$ which is $s$-harmonic  in $\O$,
$$
|D^{\g} u(x)|\leq C \, \d_{\O}^{2s-|\g|}(x)\,\int_{|y|\geq  \frac{\d_{\O}(x)}{4}}{|u(y)|}{|y|^{-N-2s}}\, dy \qquad \textrm{ for all } x\in\O,
$$
where $\d_{\O}(x)=\textrm{dist}(x,\R^N\setminus\O)$.
\end{corollary}
\proof
By assumption, $\d_{\O}(x)<\infty$   for every  $x\in\O$. Since  $B(x,\d_{\O}(x))\subset\O$, applying Lemma \ref{lem:Est-harm} to the function $y\mapsto u(y+x)$ and $R=\d_{\O}(x)$, we get the desired result.$\square$

\bigskip
\noindent
\textbf{Proof of Theorem \ref{th:ClassHArm}. }\\
\noindent
  Let $x\in\R^N$ and $r>0$. We   apply Corollary \ref{Cor:est_u}  with $\O=B(x,r)$  and   $|\g|\geq 2s$. Then, letting $r\to\infty$, we get    $|D^{\g} u(x) |=0$ for every $|\g|\geq 2s$ and $x\in\R^N$. The proof of the theorem is thus completed.   $\square$

\begin{remark}
It is well known that there are smooth functions $u$ --- hence in $L^1_{loc}(\R^N)$ --- satisfying $\D u=0$ in $\calD'(\R^N)$ for $N\geq 2$  which are not polynomials. Therefore a natural question arises:  does there exist a larger  space of distributions, strictly containing $\calL^1_s$, where the fractional laplacian is appropriately defined and where there are nontrivial entire $s$-harmonic functions which are not affine?
\end{remark}
\noindent
\textbf{Proof of Theorem \ref{th:Liouvilles012Poly}.}
 The proof will be
done by  induction. Suppose $\ell$ is the degree of $P$.
Assume that  $\ell=0$ so that  $P$ is a constant. Let
$h\in\R^N  $ and  $u_h(x)={u(x+h )-u(x)}$. It is clear that
$u_h\in \calL^1_s$. In addition $\Ds u_h =0$. It follows from Theorem \ref{th:ClassHArm} that
$\de_{i,j}u_h(0)=0$ and therefore $ \de_{i,j}u(h)= \de_{i,j}u(0) $ for
every $h\in\R^N$. This implies that $u$ is a second order polynomial and since it belongs to $\calL^1_s$, it is affine.

 Now assume that the
result holds true for a polynomial of degree up to  $\ell\geq0$ and suppose that $\Ds u= P_{\ell+1}$, a polynomial of degree $\ell+1$.
Then, for $h\in\R^N$, using the  binomial formula we can see that $\Ds u_h=
P_{\ell,h} $, where $ P_{\ell,h}$ is a polynomial of degree $\ell$.
It follows  from  our assumption that $u_h$ is affine for any $h\in\R^N$. This again implies that 
$u$ is a second order polynomial and thus an affine function,  since it belongs to $\calL^1_s$.
$\square$\\

\noindent
 \textbf{Proof of Corollary \ref{cor:cor0}.} 
We just note that   $L^q(\R^N)\subset \calL^1_s$ for every $q\in[1,\infty]$ by H\"{o}lder's inequality.  $\square$\\

\noindent \textbf{Proof of Corollary \ref{th:uniRiesz}.}
 We define the function $\ti{u}(x)=\a_{N,s}\int_{\R^N}\frac{f(y)}{|x-y|^{N-2s}}\, dy$.   Let $f_n\in C^\infty_c(\R^N)$ be  such that $f_n\to f$ in $L^p(\R^N)$.
 Define $u_n(x)=\a_{N,s}\int_{\R^N}\frac{f_n(y)}{|x-y|^{N-2s}}\, dy$. By  the  Hardy-Littlewood-Sobolev inequality (see
\cite[Theorem 4.3]{LL}), we have $u_n\to \ti{u}$ in $L^{\frac{Np}{N-2sp}}(\R^N)$. In particular, $u_n\to\ti{u}$ in $\calL^1_s$ by H\"{o}lder's inequality. Thanks to \cite[Lemma 5.3]{BB}, we have    $\Ds u_n= f_n$ in $\calD'(\R^N)$. Passing to the limit as $n\to \infty$,  we deduce that 
  $\Ds \ti{u}= f$ in $\calD'(\R^N)$.  Finally, if $u\in L^{\frac{Np}{N-2sp}}(\R^N) $ is an arbitrary solution to $\Ds u=f$ in $\calD'(\R^N)$ then $\Ds (u-\ti{u})=0$ in $\calD'(\R^N)$. We thus conclude, from Corollary \ref{cor:cor0}, that $u=\ti{u}$. $\square$\\

    \label{References}

\end{document}